\pdfoutput=1

\documentclass[12pt]{amsart}

\usepackage{color}
\usepackage{graphicx}
\usepackage[mathcal, mathscr]{euscript}
\usepackage{bbm}
\usepackage{float}
\usepackage{amsmath,amsthm,amssymb,amsfonts,amscd,epsfig,latexsym,graphicx,textcomp}
\usepackage{tikz-cd}
\usepackage{psfrag}
\usepackage[all]{xy}
\xyoption{pdf}
\usepackage{stackengine}
\usepackage{comment}
\usepackage{hhline}
\usepackage{diagbox}
\usepackage{float}
\usepackage{caption}
\usepackage{eufrak}
\usepackage{slashed}
\usepackage{amsmath, amssymb, graphics, setspace}

\usepackage{listings}

\usepackage{tikz}
\usepackage{stmaryrd}

\restylefloat{figure}

\usepackage{thmtools}
\usepackage{thm-restate}

\usepackage{hyperref}

\usepackage{cleveref}

\newtheorem{theorem}{Theorem}[section]
\newtheorem{lemma}[theorem]{Lemma}
\newtheorem{corollary}[theorem]{Corollary}

\newtheorem{question}[theorem]{Question}

\theoremstyle{definition}
\newtheorem{definition}[theorem]{Definition} 
\newtheorem{example}[theorem]{Example} 

\theoremstyle{remark}
\newtheorem{remark}[theorem]{Remark}
\numberwithin{equation}{section}

\newcommand{\IIDiag}{\raisebox{-0.33\height}{\includegraphics[scale=0.25]{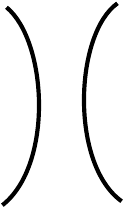}}}

\newcommand{\XDiag}{\raisebox{-0.33\height}{\includegraphics[scale=0.25]{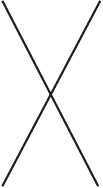}}}
\newcommand{\PMEdgeDiag}{\raisebox{-0.33\height}{\includegraphics[scale=0.25]{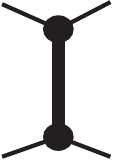}}}

\newcommand{\OneBuckleII}{\raisebox{-0.33\height}{\includegraphics[scale=0.25]{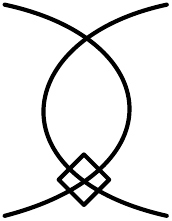}}}
\newcommand{\TwoBuckleII}{\raisebox{-0.33\height}{\includegraphics[scale=0.25]{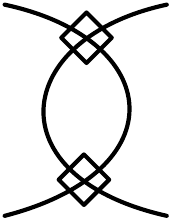}}}
\newcommand{\Buckle}{\raisebox{-0.33\height}{\includegraphics[scale=0.25]{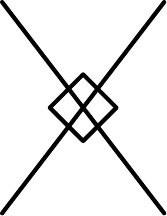}}}
\newcommand{\Node}{\raisebox{-0.33\height}{\includegraphics[scale=0.25]{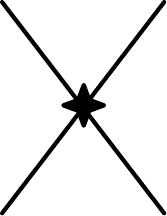}}}
\newcommand{\Virtual}{\raisebox{-0.33 \height}{\includegraphics[scale=0.25]{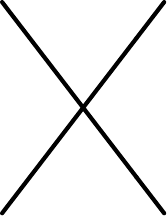}}}
\newcommand{\VirtualTwo}{\raisebox{-0.33 \height}{\includegraphics[scale=0.25]{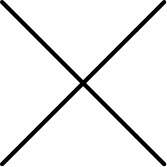}}}
\newcommand{\IINodeL}{\raisebox{-0.33\height}{\includegraphics[scale=0.25]{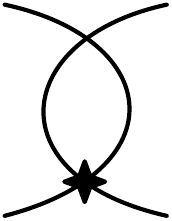}}}
\newcommand{\IINodeU}{\raisebox{-0.33\height}{\includegraphics[scale=0.25]{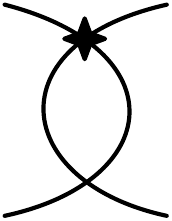}}}
\newcommand{\IINodeLU}{\raisebox{-0.33\height}{\includegraphics[scale=0.25]{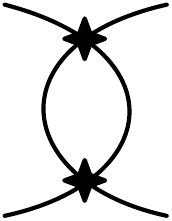}}}

\newcommand{\Pentagon}{\raisebox{-0.4 \height}{\includegraphics[scale=0.25]{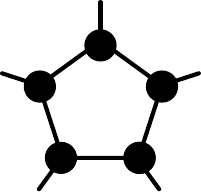}}}
\newcommand{\DRZero}{\raisebox{-0.4 \height}{\includegraphics[scale=0.25]{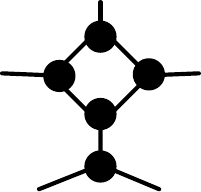}}}
\newcommand{\SMZero}{\raisebox{-0.4 \height}{\includegraphics[scale=0.25]{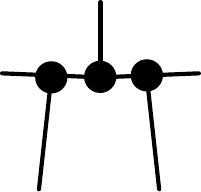}}}
\newcommand{\SMOne}{\raisebox{-0.4 \height}{\includegraphics[scale=0.25]{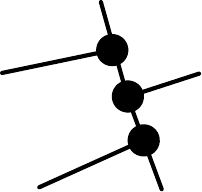}}}
\newcommand{\SMTwo}{\raisebox{-0.4 \height}{\includegraphics[scale=0.25]{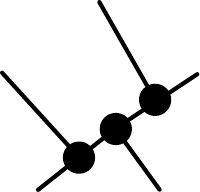}}}
\newcommand{\SMThree}{\raisebox{-0.4 \height}{\includegraphics[scale=0.25]{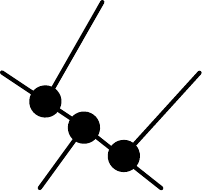}}}
\newcommand{\SMFour}{\raisebox{-0.4 \height}{\includegraphics[scale=0.25]{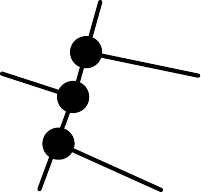}}}

\newcommand{\Quad}{\raisebox{-0.4 \height}{\includegraphics[scale=0.25]{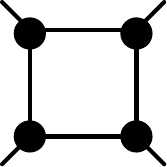}}}
\newcommand{\QuadBD}{\raisebox{-0.4 \height}{\includegraphics[scale=0.25]{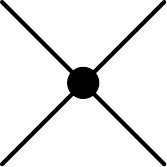}}}
\newcommand{\HH}{\raisebox{-0.4 \height}{\includegraphics[scale=0.25]{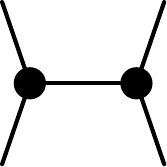}}}
\newcommand{\I}{\raisebox{-0.4 \height}{\includegraphics[scale=0.25]{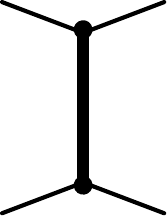}}}
\newcommand{\PlainI}{\raisebox{-0.4 \height}{\includegraphics[scale=0.25]{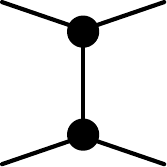}}}

\newcommand{\EqualTwo}{\raisebox{-0.4 \height}{\includegraphics[scale=0.25]{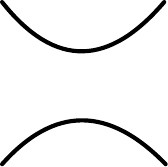}}}
\newcommand{\II}{\raisebox{-0.4 \height}{\includegraphics[scale=0.25]{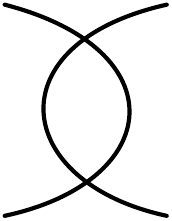}}}
\newcommand{\IITwo}{\raisebox{-0.4 \height}{\includegraphics[scale=0.25]{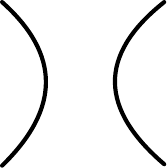}}}

\newcommand{\FourVRA}{\raisebox{-0.33\height}{\includegraphics[scale=0.5]{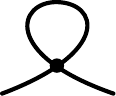}}}
\newcommand{\IDot}{\raisebox{-0.33\height}{\includegraphics[scale=0.5]{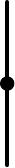}}}
\newcommand{\FourVRB}{\raisebox{-0.33\height}{\includegraphics[scale=0.5]{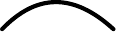}}}
\newcommand{\Bubble}{\raisebox{-0.4 \height}{\includegraphics[scale=0.5]{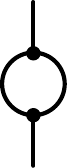}}}
\newcommand{\Vertical}{\raisebox{-0.4 \height}{\includegraphics[scale=0.5]{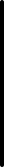}}}

\newcommand{\BUTriangle}{\raisebox{-0.4 \height}{\includegraphics[scale=0.5]{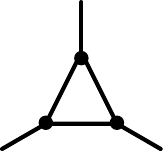}}}
\newcommand{\TrivalentVert}{\raisebox{-0.4 \height}{\includegraphics[scale=0.5]{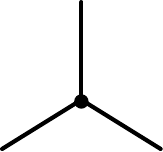}}}

\newcommand{\TrivalentVertTwist}{\raisebox{-0.4 \height}{\includegraphics[scale=0.5]{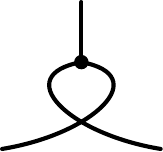}}}
\newcommand{\VertexZero}{\raisebox{-0.4 \height}{\includegraphics[scale=0.5]{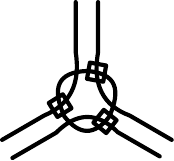}}}
\newcommand{\VertexOne}{\raisebox{-0.4 \height}{\includegraphics[scale=0.5]{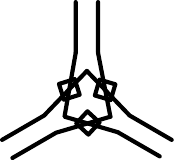}}}

\usepackage{xpatch}
\makeatletter   
\xpatchcmd{\@tocline}
{\hfil\hbox to\@pnumwidth{\@tocpagenum{#7}}\par}
{\ifnum#1<0\hfill\else\dotfill\fi\hbox to\@pnumwidth{\@tocpagenum{#7}}\par}
{}{}
\makeatother    

\makeatletter
 \def\l@subsection{\@tocline{2}{0pt}{4pc}{6pc}{}}
\def\l@subsubsection{\@tocline{3}{0pt}{8pc}{8pc}{}}
 \makeatother

\begin{document}

\title{New relations for the Penrose polynomial}

\thanks{}

\author{Scott Baldridge}
\address{Department of Mathematics, Louisiana State University,
Baton Rouge, LA}
\email{baldridge@math.lsu.edu}

\author{Ben McCarty}
\address{Department of Mathematical Sciences, University of Memphis,
Memphis, TN}
\email{ben.mccarty@memphis.edu}

\date{}

\begin{abstract} 
We introduce two new relations involving the pentagon and the quadrilateral for the evaluation of the Penrose polynomial at $n=4$ that is proven using a new type of ribbon graph polynomial.  Additionally, we extend several relations for the evaluation of the Penrose polynomial at $n=3$ to all $n$.  
\end{abstract}

\maketitle

\section{Introduction}
In 1971, Penrose~\cite{Penrose} introduced a diagrammatic method for studying abstract tensor systems. As a result, he derived several formulas for computing the number of $3$-edge colorings of a planar trivalent graph. The most notable of these formulas gave rise to what is now known as the Penrose polynomial. This polynomial can be computed from a \emph{perfect matching graph} $\Gamma_M$, which is a ribbon graph $\Gamma$ associated with a trivalent graph $G(V,E)$ together with a perfect matching $M \subset E$. By a ribbon graph, we mean the closure of a small neighborhood of the $1$-skeleton of a CW-complex of a closed surface $\overline{\Gamma}$, together with the $1$-skeleton (see~\cite{BM-Color,BKR} for a detailed treatment). The Penrose polynomial of $\Gamma_M$ is defined using the bracket on the diagram of it,
\[
\left[ \PMEdgeDiag \right]_n = \left[ \IIDiag \right]_n - \left[ \XDiag \right]_n,
\]
\[
\left[ \bigcirc \right]_n = n,
\]
where the bold edge in the first equation denotes an edge in the perfect matching (cf.~\cite{Aigner,BM-Color,BM-Reduce,BKM-TFC,EMM}). We refer to $n$ above as the \emph{loop value}.

While the Penrose polynomial generally depends on both the ribbon graph and the perfect matching, Penrose~\cite{Penrose} showed that its evaluation at $n=3$ is independent of the perfect matching. In fact, Penrose showed that its evaluation is always the number of $3$-edge colorings of the underlying  graph when it  is planar. He then described several relations that hold when the loop value is $n=3$ (cf. \Cref{fig:Penrose-Orig-Relations}).
\begin{figure}[H]
  \centering
  \includegraphics[scale=0.44]{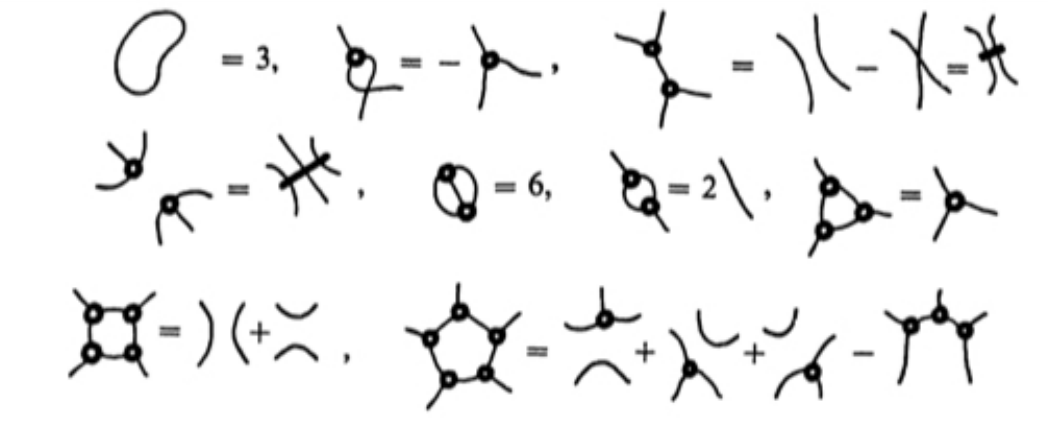}
  \caption{Relations from Penrose's original paper~\cite{Penrose}. In each case, the relation involves the bracket evaluated at $n=3$.}
  \label{fig:Penrose-Orig-Relations}
\end{figure}

To remove the dependency on the perfect matching and obtain a polynomial invariant of the ribbon graph, and to extend the construction beyond  trivalent graphs, we \emph{blowup} every vertex of the original ribbon graph $\Gamma$, in which a vertex of degree $d$ is replaced by a $d$-cycle. For example, when $d=3$, the vertex is replaced by a $3$-cycle, as shown below.
\begin{center}
  \includegraphics[scale=0.05]{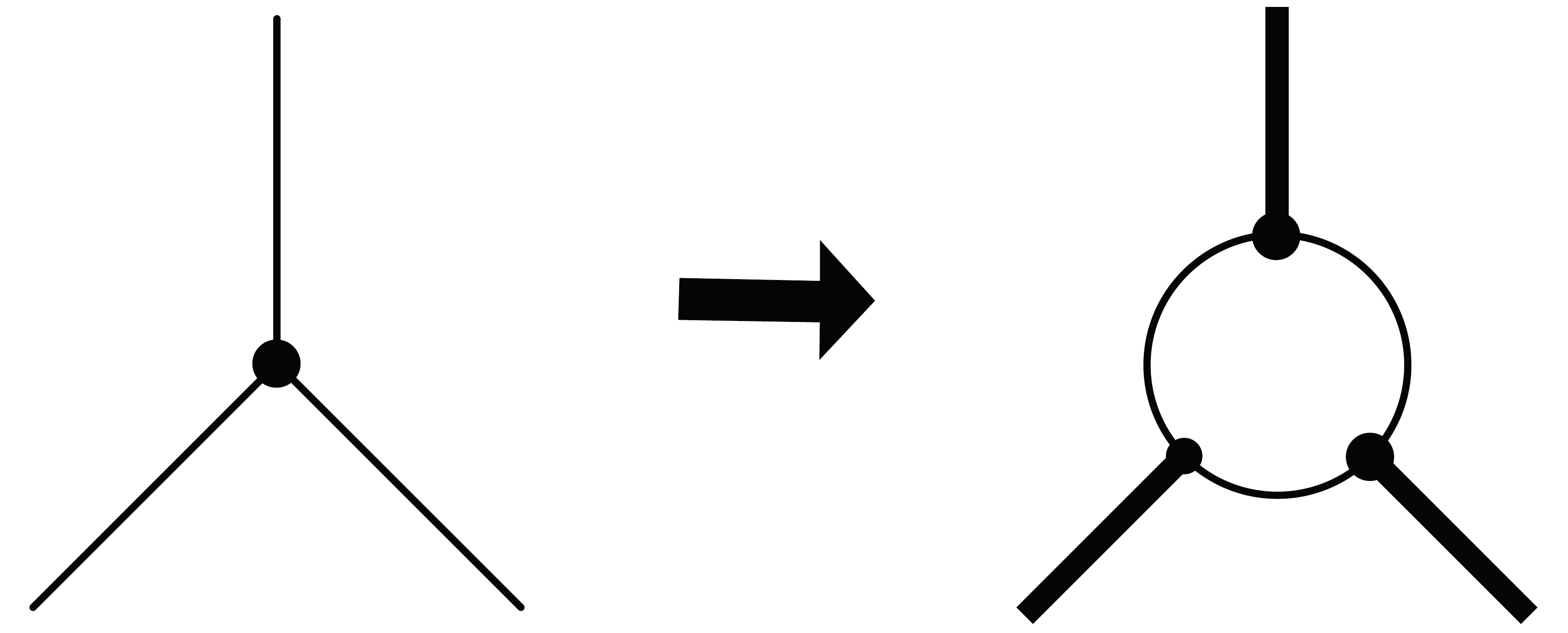}
\end{center}
The resulting blowup graph is trivalent, even when the original ribbon graph is not, and it has a canonical perfect matching given by the edge set of the original graph. We denote the resulting perfect matching graph by $\Gamma_E^\flat$. The Penrose polynomial of the ribbon graph $\Gamma$ is defined as the Penrose polynomial of $\Gamma_E^\flat$.

Working with the blowup, one can show that some, but not all, of Penrose's relations (cf. \Cref{fig:Penrose-Orig-Relations}) hold for the blowup and can be extended to relations that hold for all $n$ (cf.~\Cref{thm:polynomialRelations}). Others, such as those for the pentagon and quadrilateral, are significantly more challenging to extend. For 55 years, no such relations were known for the pentagon or quadrilateral  for a loop value of $n>3$. In this paper, we employ a technique inspired by Kauffman's multi-virtuals and the chromatic polynomial (cf.~\cite{BKM-TFC,EMM,KPrep}) to reduce the calculation of an infinite number of graphs down to a finite number, as described in \Cref{thm:capRelation}. This corollary is the key theoretical input for this paper. It  is used to prove the following relations of the Penrose polynomial in \Cref{thm:4gonRelation,thm:5gonRelation} for $n=4$:

\medskip

\[
\left[ \Quad \right]_4 = 2 \left[ \HH \right]_4 + 2 \left[ \PlainI \right]_4 + 4 \left[ \IITwo \right]_4 + 4 \left[ \EqualTwo \right]_4 + 4 \left[ \VirtualTwo \right]_4 - 4 \left[ \QuadBD \right]_4,
\]

\medskip

\[
\left[ \Pentagon \right]_4 = \frac{3}{2} \left[ \DRZero \right]_4 - 4 \left[ \SMZero \right]_4 - \left[ \SMOne \right]_4 + 2 \left[ \SMTwo \right]_4 + 2 \left[ \SMThree \right]_4 - \left[ \SMFour \right]_4.
\]

\medskip
\noindent In this paper, the outer edges emanating from the quadrilateral and pentagon are always assumed to be unique in the ribbon graph that contains them.

The relation for the pentagon suggests a new way to prove  the Four Color Theorem.  This is especially striking because Kempe's original strategy failed at the point where one encounters a pentagonal configuration.  The pentagon relation above raises the possibility of reviving that idea: if one could show that the left-hand side is necessarily positive when evaluated at $n=4$, then by \cite{BM-Color} (see also \cite{Aigner}) the Penrose polynomial would be positive at $4$, and hence the graph would admit a four-face-coloring.

We have aimed to introduce terms as they arise, but readers looking for a more comprehensive treatment should consult \cite{BaldCohomology,BKR,BLM,BM-Color,BM-Reduce,BKM-TFC}.   In \Cref{sec:polys} we define the polynomials used.  \Cref{sec:caps} describes the needed framework for proving the new relations, and \Cref{sec:Relations} describes the proof itself.

\section{The Polynomials}\label{sec:polys}

First, we formalize the definition of the Penrose polynomial from \cite{BM-Color} based on the discussion above (see also \cite{BaldCohomology}).  Unless otherwise noted, all graphs $G$ are connected and trivalent.

\begin{definition}\label{defn:penrose_poly}
Let $\Gamma_M$ be a perfect matching graph for the pair $(G,M)$.  Then the {\em Penrose polynomial}, denoted $\left[\Gamma_M\right]_n$,  is found by recursively applying the bracket $$\left[ \I \right]_{\! n} = \left[ \IIDiag  \right]_{\! n}  - \left[ \XDiag \right]_{\! n}$$ to perfect matching edges of $\Gamma_M$ and setting the value of immersed loops to $\left[ \bigcirc \right]_{\! n} = n$.  For $\Gamma$ a  ribbon graph of $G$, define $[\Gamma]_n:=[\Gamma^\flat_E]_n$.
\end{definition}

The recursive relation in~\Cref{defn:penrose_poly} naturally leads to a \emph{hypercube of states}, defined as follows. Given a trivalent graph $G(V,E)$ with a perfect matching $M$, the number of vertices $|V|$ is even, and the number of edges in $M$ is $\ell = |V|/2$. Label and order these edges as $M = \{e_1, e_2, \dots, e_\ell\}$. Let $\Gamma_M$ denote a perfect matching graph associated with $(G,M)$, presented as a diagram in the plane. For each perfect matching edge $e_i \in M$, resolve it in one of two ways according to the smoothings: replace a neighborhood of $e_i$ in $\Gamma_M$ with $\IIDiag$, called a {\em $0$-smoothing}, or $\XDiag$, called a {\em $1$-smoothing}. The resulting set of immersed circles in the plane is called a \emph{state} of $\Gamma_M$. By construction, each state can be regarded as a ribbon graph, where a $0$-smoothing corresponds to an untwisted ribbon and a $1$-smoothing to a twisted ribbon. We refer to this ribbon graph as a \emph{state graph}.

There are $2^\ell$ states of $\Gamma_M$, each indexed by an $\ell$-tuple $\alpha = (\alpha_1, \dots, \alpha_\ell) \in \{0,1\}^\ell$, where $\alpha_i$ indicates the type of smoothing for edge $e_i$. For a state $\Gamma_\alpha$, the edge $e_i$ is resolved by an $\alpha_i$-smoothing. Define $|\alpha| = \sum_{i=1}^\ell \alpha_i$, and organize the states into columns based on the value of $|\alpha|$ (see~\Cref{fig:P3Hypercube}).
\begin{figure}[H]
  \centering
  \includegraphics[scale=0.75]{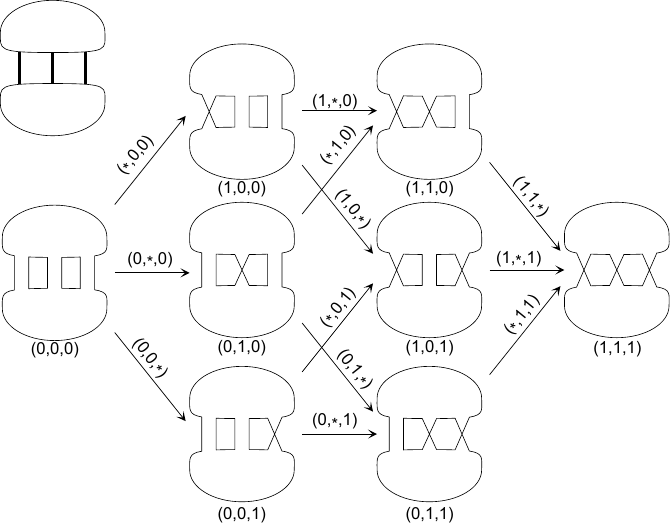}
  \caption{A hypercube of states.}
  \label{fig:P3Hypercube}
\end{figure}

\subsection{A new extension of the Penrose polynomial}
Theorems~E and~7.2 in~\cite{BM-Color} establish that the evaluation of the Penrose polynomial at $n$ is determined by a signed count of proper colorings of state graphs. Regarding each state graph as a ribbon graph in its own right, this implies that the Penrose polynomial is a signed sum of the chromatic polynomials of the duals of these state graphs. We now introduce a bracket inspired by this property.

\begin{definition}\label{defn:BMBracket}
Let $\Gamma_M$ be a perfect matching graph for the pair $(G,M)$.  Then the {\em BM (chromatic) polynomial}, denoted $\llbracket\Gamma_M \rrbracket_n$,  is computed by recursively applying the bracket
$$\bigg \llbracket \I \bigg \rrbracket_{\! n} = \bigg \llbracket \OneBuckleII  \bigg \rrbracket_{\! n}  - \bigg \llbracket \Buckle\bigg \rrbracket_{\! n}$$
$$\bigg \llbracket \Buckle \bigg \rrbracket_{\! n} = \bigg \llbracket \Virtual  \bigg \rrbracket_{\! n}  - \bigg \llbracket \Node \bigg \rrbracket_{\! n}$$
 to perfect matching edges of $\Gamma_M$ and setting the value of immersed loops to $ \llbracket \bigcirc  \rrbracket_{\! n} = n$.  Here, the node, \Node means that the two arcs are treated as a single circle, and referred to as \emph{fused}. 
\end{definition}

\begin{remark}
The squares (or diamonds) in~\Cref{defn:BMBracket} were previously introduced in~\cite{BM-Reduce} as \emph{buckles} and \emph{clasps}. In that paper, the term ``buckle'' referred only to interactions at a $1$-smoothing site. For simplicity, we use the term ``buckle'' for both in this paper, but the meaning of the symbol is the same as in the earlier paper.
\end{remark}

Applying the first relation of \Cref{defn:BMBracket} produces the same hypercube of states as that of the Penrose polynomial---each state of which is a set of immersed circles, except for the additional squares decorating the virtual crossings at the smoothing sites.  Call this the hypercube of states for the BM polynomial. When referring to a state graph, we let the context determine whether the smoothing sites are decorated with squares, as in the computation of the BM polynomial, or not, as in the computation of the Penrose polynomial. 
 
Note that this bracket is different than the Penrose-Kauffman bracket defined in \cite{BKM-TFC}, which was used to compute the total face color polynomial.  The PK bracket computes the Poincar\'e polynomial of a certain homology theory (cf. \cite{BM-Color}) whereas the bracket for the BM polynomial computes the Euler characteristic of that homology (see \Cref{thm:EAnalog}).  In fact, one can show that if you replace each circle in one of the states above by a vertex and each buckle by an edge between the vertices representing those circles, the bracket for the BM polynomial computes the chromatic polynomial of that abstract graph (see \Cref{lem:BMState}).
 
Our first observation is that, when applied to the blowup of a ribbon graph with its canonical perfect matching, the BM polynomial equals the Penrose polynomial, which is also the Euler characteristic of the homology in \cite{BM-Color}. The squares, when resolved using the second relation in \Cref{defn:BMBracket},  give the following:
\begin{align*}
\bigg \llbracket \hspace{.1cm} \I \hspace{.1cm} \bigg \rrbracket_{\! n} &=  \bigg \llbracket  \hspace{.1cm} \OneBuckleII  \hspace{.1cm} \bigg \rrbracket _{\! n} - \bigg \llbracket  \hspace{.1cm} \Buckle  \hspace{.1cm} \bigg \rrbracket _{\! n}\\
&= \bigg \llbracket  \hspace{.1cm} \II  \hspace{.1cm} \bigg \rrbracket _{\! n} - \bigg \llbracket  \hspace{.1cm} \IINodeL  \hspace{.1cm} \bigg \rrbracket _{\! n}- \bigg \llbracket  \hspace{.1cm} \Virtual  \hspace{.1cm} \bigg \rrbracket _{\! n} + \bigg \llbracket  \hspace{.1cm} \Node  \hspace{.1cm} \bigg \rrbracket _{\! n}\\
&=  \bigg \llbracket  \hspace{.1cm} \II  \hspace{.1cm} \bigg \rrbracket _{\! n} - \bigg \llbracket  \hspace{.1cm} \Virtual  \hspace{.1cm} \bigg \rrbracket _{\! n}\\
&=  \left[  \hspace{.1cm} \I  \hspace{.1cm} \right]_{\! n}. 
\end{align*}
The two terms with nodes in the second line are  treated as the same ``fused'' circle, and therefore cancel. Thus, we have proven the following theorems.

\begin{theorem}\label{thm:BracketRelation}
Let $\Gamma$ be a ribbon graph for a trivalent graph $G(V,E)$.  Then,
$$\llbracket \Gamma^\flat_E  \rrbracket_n = \left[\Gamma^\flat_E \right]_n,$$
and hence $\llbracket \Gamma \rrbracket_n = \left[\Gamma\right]_n$.
\end{theorem}

\begin{theorem}
Given $\Gamma_M$ a perfect matching graph for the pair $(G,M)$ without any buckles, then the Penrose polynomial and the BM polynomial are equal.  That is, 
$$\llbracket \Gamma_M \rrbracket_{\! n} = \left[ \Gamma_M \right]_{\! n}.$$
\end{theorem}

One might question the sanity of introducing a new bracket only to say that it yields the same result as the Penrose bracket, especially when it involves more computations.  However, the second bracket of the BM polynomial in \Cref{defn:BMBracket}, which we call the \emph{buckle bracket}, allows us to derive a BM polynomial of {\em each} state graph. 

Fortunately, in such calculations, not all buckle brackets need to be expanded. Observe that whenever two or more buckles appear on two arcs, all but one can be eliminated:
\begin{align*}
\bigg \llbracket \hspace{.1cm} \TwoBuckleII \hspace{.1cm} \bigg \rrbracket_{\! n} &=  \bigg \llbracket  \hspace{.1cm} \II  \hspace{.1cm} \bigg \rrbracket _{\! n} - \bigg \llbracket  \hspace{.1cm} \IINodeL  \hspace{.1cm} \bigg \rrbracket _{\! n}- \bigg \llbracket  \hspace{.1cm} \IINodeU  \hspace{.1cm} \bigg \rrbracket _{\! n}+ \bigg \llbracket  \hspace{.1cm} \IINodeLU  \hspace{.1cm} \bigg \rrbracket _{\! n}\\
&=  \bigg \llbracket  \hspace{.1cm} \II  \hspace{.1cm} \bigg \rrbracket _{\! n} - \bigg \llbracket  \hspace{.1cm} \IINodeL  \hspace{.1cm} \bigg \rrbracket _{\! n}\\
&=  \bigg \llbracket  \hspace{.1cm} \OneBuckleII  \hspace{.1cm} \bigg \rrbracket_{\! n}. 
\end{align*}
Thus, we obtain the following.

\begin{lemma}
Let $\Gamma_M$ be a perfect matching ribbon graph for a trivalent graph $G(V,E)$ with perfect matching $M\subset E$, and let $\Gamma_\alpha$ be a state graph for the BM polynomial.  For each pair of circles in $\Gamma_\alpha$ that interact in some number of buckles, all but one of the buckles may be removed when applying the buckle relation.
\end{lemma}

Unlike the Penrose bracket, which for a state simply replaces each circle with a factor of $n$, the BM polynomial of the state counts proper $n$-face colorings of the ribbon graph associated to the state, as demonstrated by the following lemma.

\begin{lemma}\label{lem:BMState}
Let $\Gamma$ be a ribbon graph for a trivalent graph $G(V,E)$, and let $\Gamma_\alpha$ be a state graph for the BM polynomial of the blowup $\Gamma_E^\flat$.  Then the BM polynomial of $\Gamma_\alpha$ is the number of proper $n$-face colorings of the closed surface $\overline \Gamma_\alpha$, or, using the harmonic colorings of the state from \cite{BM-Color},
$$\llbracket \Gamma_\alpha \rrbracket_{\! n} = \dim \widehat{\mathcal{CH}}_n(\Gamma_\alpha).$$
In particular, if a buckle is on any single circle in a state graph $\Gamma_\alpha$, then $\llbracket \Gamma_\alpha \rrbracket_{\! n} = 0$.
\end{lemma}

\begin{proof}
Notice that $\Gamma_\alpha$ may be regarded as a ribbon graph in its own right, where $G$ corresponds to the $1$-skeleton of the surface and the circles of the state correspond to the boundaries of faces (see Definition 6.14 in \cite{BM-Color}). Taking the ribbon graph dual, the buckle relation of the BM polynomial corresponds to a deletion-contraction relation.  Thus, the BM polynomial of the state graph is the chromatic polynomial of the dual graph.  Since vertex colorings of the dual correspond to face colorings, and $\dim \widehat{\mathcal{CH}}_n(\Gamma_\alpha)$ is a count of the number of face colorings of the ribbon graph (thought of without the buckles), the result follows.
\end{proof}

\begin{remark}
\Cref{lem:BMState} extends to the case of a perfect matching graph $\Gamma_M$ that is not necessarily a blowup.  We formulate the lemma here for the blowup graph to ensure that the circles in the state correspond directly to the faces of the graph and thus, harmonic colorings correspond to proper face colorings of the associated surface.  In \cite{BKM-TFC}, we provide an interpretation of harmonic colorings for the case of a perfect matching ribbon graph.
\end{remark}

\Cref{lem:BMState} highlights a key distinction between the Penrose polynomial and the BM polynomial. Specifically, when computing the BM polynomial of a perfect matching ribbon graph with certain edges decorated by buckles, the result may differ from the Penrose polynomial. We first illustrate this with an example.

\begin{example}\label{ex:twistedDoubleTheta}
For the perfect matching ribbon graph in~\Cref{fig:twistedDoubleTheta}, if we ignore the buckle, and treat the crossing as an ordinary virtual crossing, the Penrose polynomial is $n^2 - n$.
\begin{figure}[H]
  \centering
  \includegraphics[scale=0.5]{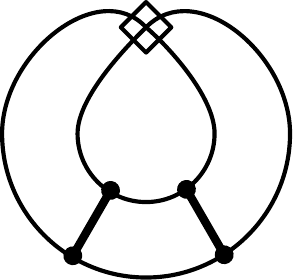}
  \caption{A double theta graph with a twist.}
  \label{fig:twistedDoubleTheta}
\end{figure}
\noindent However, examining the four state graphs in this calculation, one observes that every proper coloring of each state graph must assign the same color to the two intersecting arcs at the top: the buckle is always on a single circle in each state. By~\Cref{lem:BMState}, if we include the buckle in our computation of the BM polynomial, we find that BM polynomial is zero.

If we examine the eight states for the $3$-prism in~\Cref{fig:P3}, we find that four of them form a subcube corresponding to the states of the ribbon graph in \Cref{fig:twistedDoubleTheta}.
\begin{figure}[H]
  \centering
  \includegraphics[scale=0.5]{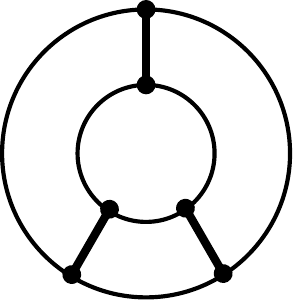}
  \caption{The $3$-prism, or the blowup of a $\theta$-graph.}
  \label{fig:P3}
\end{figure}
\noindent In this case, since the BM polynomial is $0$, none of the four states corresponding to \Cref{fig:twistedDoubleTheta} contribute any colorings to the Penrose polynomial of the $3$-prism, and the proper colorings are only supported on states that have a zero smoothing at the top-most perfect matching edge.  

\end{example}

In general, \Cref{lem:BMState} and Proposition~6.15 of~\cite{BM-Color} imply that the BM polynomial is the signed sum of the proper $n$-face colorings of all state graphs. Throughout the remainder of this paper, we work with state graphs of the BM polynomial. This necessitates the use of ribbon graphs whose virtual crossings may have buckles, as illustrated  in \Cref{fig:twistedDoubleTheta}. The analysis above leads to the following theorem.

\begin{theorem}\label{thm:EAnalog}
Let $\Gamma_M$ be a perfect matching graph for the pair $(G,M)$, possibly including interactions between edges of $E\setminus M$ marked with a buckle. Let $\Gamma_\alpha$ be the state graph for the state $\alpha \in \{0,1\}^{|M|}$ in the hypercube of states for $\Gamma_M$. Then, for $n$ a positive integer,
$$ \llbracket \Gamma_M  \rrbracket_{\! n} = \sum_{\alpha \in \{0,1\}^{|M|}} (-1)^{|\alpha|}\llbracket \Gamma_\alpha\rrbracket_{\! n}.$$

\end{theorem}

Moreover, in proving~\Cref{lem:BMState}, we observed that the effect of a buckle on a pair of interacting arcs is to ensure that the BM polynomial counts only colorings where the two arcs at a buckle are assigned different colors. This yields the following:

\begin{corollary}\label{thm:ContrastPenrose}
Let $\Gamma_M$ be a perfect matching graph for the pair $(G,M)$, possibly including interactions between edges of $E\setminus M$ marked with a buckle.  Then, for $n$ a positive integer,
\[
\llbracket \Gamma_M \rrbracket_n = \sum_{\alpha} (-1)^{|\alpha|} \# \left\{ \mbox{$n$-colorings of the circles of $\Gamma_\alpha$ with different colors at each buckle} \right\}.
\]
\end{corollary}

\begin{remark} When $\Gamma_M$ is the blow up of a ribbon graph $\Gamma$ with $M=E$, as in \Cref{lem:BMState}, the $n$-colorings of the circles of $\Gamma_\alpha$ correspond to proper $n$-face colorings of the state graph $\Gamma_\alpha$.
\end{remark}

\medskip

\Cref{thm:ContrastPenrose} reveals the true utility of the BM polynomial: it can be viewed as a relative version of the Penrose polynomial. To make this precise, we introduce the notion of a capped configuration in the next section.

\section{Caps, configurations, and a relative bracket}\label{sec:caps}
In Section 4 of \cite{BM-Reduce}, we described the theory that enables the reduction of coloring problems for configurations to the study of a finite set of ``caps," which can be thought of as a ``basis" for computing the BM polynomial, and hence, the Penrose polynomial.  We adopt the same approach here.  In this subsection, we recall the basics of caps and configurations through an example and refer the reader to \cite{BM-Reduce} for a detailed treatment.  

\begin{figure}[H]
\includegraphics[scale=0.3]{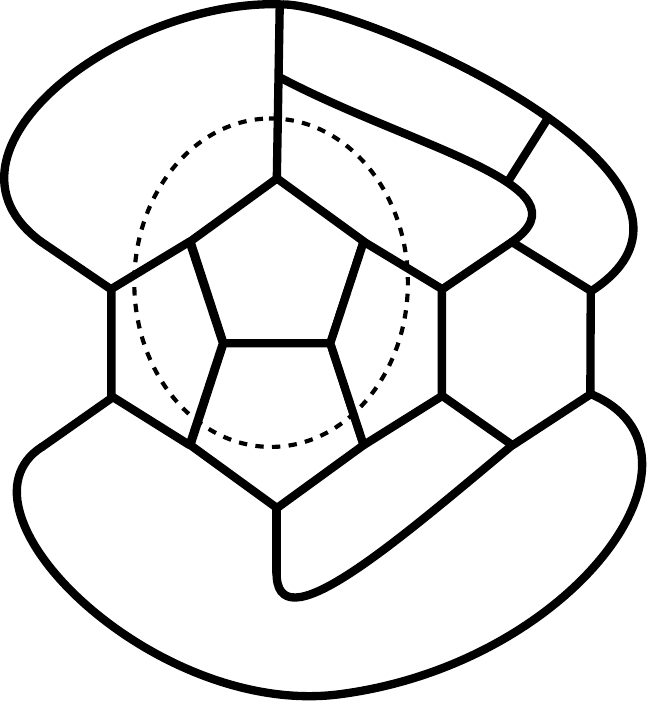}
\caption{An example of a graph that contains a pentagonal face.}\label{fig:capExample}
\end{figure}

Notice that the graph in \Cref{fig:capExample} contains a pentagonal face (circled by a dotted line).  Our plan is to separate the the ribbon graph into two parts:  one inside the dotted circle, referred to as the \emph{configuration} and one outside the dotted circle, which can be captured by a set of caps.  The edges that meet the boundary of the configuration are called \emph{spokes}.  Throughout, we assume that the spokes outside the dotted circle are unique.

Consider a state graph $\Gamma_\alpha$ such as the one shown on the left of \Cref{fig:capExample1}.  The pentagonal face can be excised as shown, with any circles outside the disk removed (see \Cref{fig:capExample1}).  
\begin{figure}[h]
\includegraphics[scale=0.3]{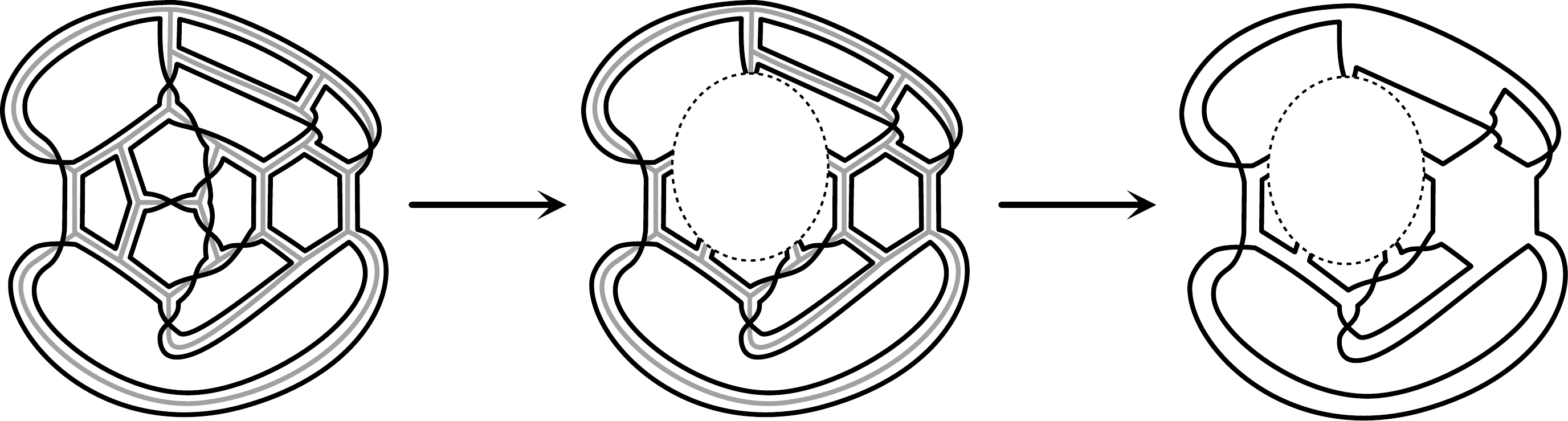}
\caption{Excising a pentagonal face and creating a cap.}\label{fig:capExample1}
\end{figure}
We then simplify the diagram by marking one interaction site with a buckle for each pair of interacting arcs, excluding any arcs that interact along a spoke (these are addressed when considering smoothings of the configuration). The result is a diagram of a \emph{cap} (see \Cref{fig:capExample2}).  
\begin{figure}[h]
\includegraphics[scale=0.3]{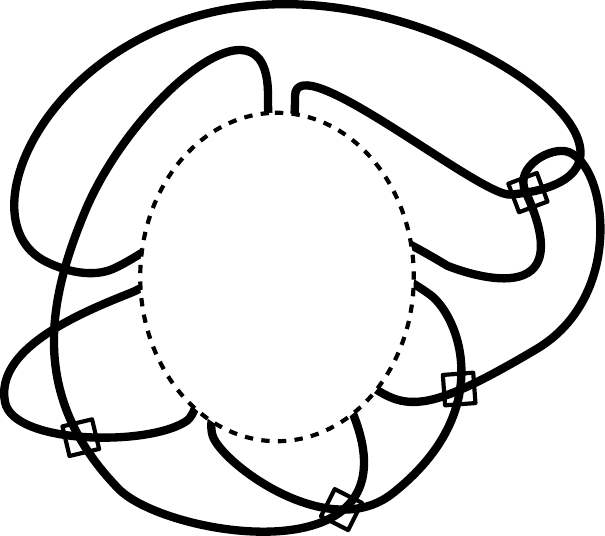}
\caption{A cap $C$.}\label{fig:capExample2}
\end{figure}
Inserting a configuration, such as a pentagon $P$, into a cap, denoted $C$, produces a perfect matching graph decorated by interaction sites in the cap. We denote this by $C\#P$ and refer to as a \emph{capped configuration} (see \Cref{fig:cappedConfig}).

\begin{figure}[h]
\includegraphics[scale=0.3]{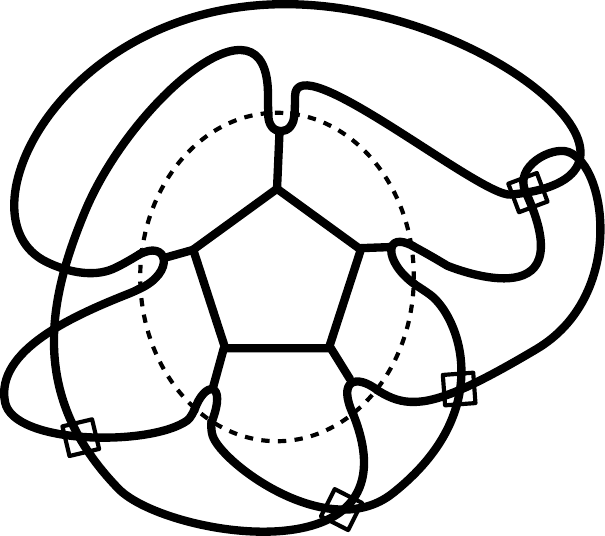}
\caption{A capped configuration $C\#P$.}\label{fig:cappedConfig}
\end{figure}

For a given number of spokes, the set of all possible caps is clearly finite: for $n$ spokes, we consider all possible ways of connecting them with $n$ arcs and then assign buckles to all possible interactions between the arcs.

Our approach is to take each cap, insert various configurations, compute the polynomials for each capped configuration, and identify relations among the polynomials that hold for every possible cap. With the Penrose polynomial, this approach is ineffective, as the Penrose recursion does not account for interactions between the cap arcs. This is where the BM polynomial has an advantage. The buckle relation of the BM polynomial enables computation of the polynomial of a capped configuration, accounting for all relevant interactions between the cap arcs. To demonstrate why our list of caps is sufficient to cover all possible exteriors, we first present a lemma and a simple example.

\begin{lemma}
Given circle, contained in the cap, that interacts with a single cap arc, one may remove the circle, at the expense of multiplying the BM polynomial by $(n-1)$.

\begin{align*}
\bigg \llbracket \hspace{.1cm} \raisebox{-0.25\height}{\includegraphics[scale=.6]{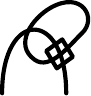} }\hspace{.1cm} \bigg \rrbracket_{\! n} &=  (n-1) \bigg \llbracket   \hspace{.1cm} \raisebox{-0.25\height}{\includegraphics[scale=.6]{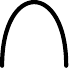} }  \hspace{.1cm} \bigg \rrbracket_{\! n}. 
\end{align*}
\end{lemma}

\begin{proof}
The proof is the calculation:
\begin{align*}
\bigg \llbracket \hspace{.1cm} \raisebox{-0.25\height}{\includegraphics[scale=.6]{SadBuckle.pdf} }\hspace{.1cm} \bigg \rrbracket_{\! n} &=  \bigg \llbracket  \hspace{.1cm} \raisebox{-0.25\height}{\includegraphics[scale=.6]{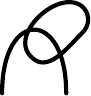} }  \hspace{.1cm} \bigg \rrbracket _{\! n} - \bigg \llbracket  \hspace{.1cm} \raisebox{-0.25\height}{\includegraphics[scale=.6]{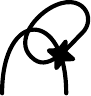} } \hspace{.1cm} \bigg \rrbracket _{\! n}\\
&=  (n-1) \bigg \llbracket  \hspace{.1cm} \raisebox{-0.25\height}{\includegraphics[scale=.6]{StLouis.pdf} }  \hspace{.1cm} \bigg \rrbracket_{\! n}. 
\end{align*}

\end{proof}

\begin{example}
For this example, we focus on two cap arcs that interact with two circles in the outer region, as shown in Figure \ref{fig:simpleCap}.  There may or may not be other cap arcs, and other circles in the outer region, and the graph.
\begin{figure}
\includegraphics[scale=.6]{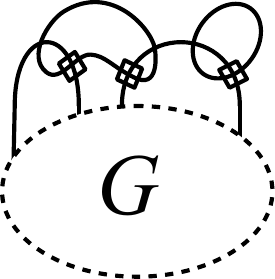} 
\caption{A simple cap.}
\label{fig:simpleCap}
\end{figure}

We then calculate:
\begin{align*}
\bigg \llbracket \hspace{.1cm} \raisebox{-0.25\height}{\includegraphics[scale=.6]{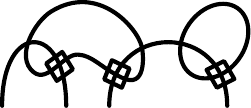} }\hspace{.1cm} \bigg \rrbracket_{\! n} &=  \bigg \llbracket  \hspace{.1cm} \raisebox{-0.25\height}{\includegraphics[scale=.6]{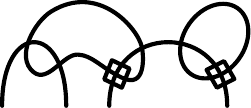} }  \hspace{.1cm} \bigg \rrbracket _{\! n} - \bigg \llbracket  \hspace{.1cm} \raisebox{-0.25\height}{\includegraphics[scale=.6]{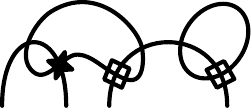} } \hspace{.1cm} \bigg \rrbracket _{\! n}\\
&=  \bigg \llbracket  \hspace{.1cm} \raisebox{-0.25\height}{\includegraphics[scale=.6]{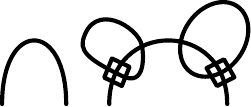} }  \hspace{.1cm} \bigg \rrbracket _{\! n} - \bigg \llbracket  \hspace{.1cm} \raisebox{-0.25\height}{\includegraphics[scale=.6]{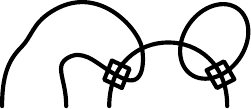} } \hspace{.1cm} \bigg \rrbracket _{\! n}\\
&=  (n-1)^2\bigg \llbracket  \hspace{.1cm} \raisebox{-0.25\height}{\includegraphics[scale=.6]{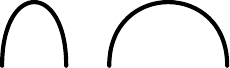} }  \hspace{.1cm} \bigg \rrbracket _{\! n} - (n-1)\bigg \llbracket  \hspace{.1cm} \raisebox{-0.25\height}{\includegraphics[scale=.6]{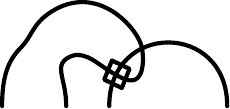} } \hspace{.1cm} \bigg \rrbracket _{\! n}\\
\end{align*}
This calculation illustrates the principle that, for any outer region, the BM polynomial can be expressed in terms of caps consisting only of arcs that begin and end at the configuration, along with their interactions, with coefficients that are polynomials in $n$.  In other words, the finite list of caps serves as a basis.  
\label{ex:simpleCap}
\end{example}

By~\Cref{thm:EAnalog}, the BM polynomial is a signed sum of the BM polynomials of all state graphs. For each state graph, we can express its BM polynomial in terms of the BM polynomials of capped states, as demonstrated in~\Cref{ex:simpleCap}. By appropriately collecting the capped states, we can express the BM polynomial in terms of the polynomials of capped configurations, as summarized in the following theorem.

\begin{theorem}\label{thm:polyFromCaps}
Given a ribbon graph $\Gamma$ containing a configuration $C$ with $k$ spokes.  Let $\mathcal{NP}_k$ be the set of all caps (not necessarily planar) on $k$ spokes.  Then the BM polynomial of $\Gamma$ can be written in terms of capped configurations.  That is,
$$\llbracket \Gamma \rrbracket_n = \sum_{C_i \in \mathcal{NP}_k} f_i (n) \llbracket C_i \# C \rrbracket_n$$
where $f_i(n)$ is a polynomial in $n$, and $\llbracket C_i \# C \rrbracket_n$ is the BM polynomial of the capped configuration.
\end{theorem}

Combining \Cref{thm:polyFromCaps} with \Cref{thm:ContrastPenrose} along with Theorem E and 7.2 of \cite{BM-Color} we obtain the following corollary:

\begin{corollary}\label{thm:capRelation}
Any relation between the BM polynomials of a set of configurations that holds for each cap individually, holds for the Penrose polynomial in general. Specifically, let $C_1, \ldots, C_\ell$ be configurations with $k$ spokes each.  For polynomials $a_i(n)$, if 
$$\sum_{i=1}^\ell a_i(n) \llbracket C_i \# Cap \rrbracket_n =0$$
holds for all caps $Cap \in \mathcal{NP}_k$, then for any ribbon graph $\Gamma$ with a $k$-spoke configuration removed, such that the spokes are unique in $\Gamma$,
$$\sum_{i=1}^\ell a_i(n) P(\Gamma\#C_i) = 0.$$
\end{corollary}

\begin{remark}
The corollary above also applies for relations at specific values of $n$ as in \Cref{thm:4gonRelation,thm:5gonRelation} below, or could be applied in general for relations that hold for all $n$, such as those in \Cref{thm:polynomialRelations}.
\end{remark}

\subsection{The Pentagon and the Quadrilateral}  As an application of \Cref{thm:capRelation}, we specialize to $n=4$ and give two new relations: one for a quadrilateral, and one for a pentagon.  

\begin{theorem}
\label{thm:4gonRelation}
Let $\Gamma$ be a ribbon graph of an abstract graph $G(V,E)$ that has a quadrilateral face whose emanating edges are all unique in $\Gamma$.  Replacing the quadrilateral with each of the configurations shown on the right, we obtain a family of ribbon graphs whose Penrose polynomials, evaluated at $n=4$, satisfy the following relation:\\
$$\left[ \hspace{.1cm} \Quad \hspace{.1cm} \right]_{\! 4} =  2 \left[ \hspace{.1cm} \HH  \hspace{.1cm} \right]_{\! 4} + 2 \left[ \hspace{.1cm} \PlainI  \hspace{.1cm} \right]_{\! 4} + 4 \left[ \hspace{.1cm} \IITwo  \hspace{.1cm} \right]_{\! 4} + 4 \left[ \hspace{.1cm} \EqualTwo  \hspace{.1cm} \right]_{\! 4} + 4\left[ \hspace{.1cm} \VirtualTwo  \hspace{.1cm} \right]_{\! 4} - 4 \left[ \hspace{.1cm} \QuadBD  \hspace{.1cm} \right]_{\! 4}.$$
\end{theorem}

\medskip

\begin{theorem}
\label{thm:5gonRelation}
Let $\Gamma$ be a ribbon graph of an abstract graph $G(V,E)$ that contains a pentagonal face whose emanating edges are all unique in $\Gamma$.  Replacing the pentagon with each of the configurations shown on the right, we obtain a family of ribbon graphs whose Penrose polynomials, evaluated at $n=4$, satisfy the following relation:\\
$$\left[ \hspace{.1cm} \Pentagon \hspace{.1cm} \right]_{\! 4} =  \frac{3}{2}\left[ \hspace{.1cm} \DRZero  \hspace{.1cm} \right]_{\! 4} - 4\left[ \hspace{.1cm} \SMZero  \hspace{.1cm} \right]_{\! 4} - \left[ \hspace{.1cm} \SMOne  \hspace{.1cm} \right]_{\! 4} + 2 \left[ \hspace{.1cm} \SMTwo  \hspace{.1cm} \right]_{\! 4} + 2 \left[ \hspace{.1cm} \SMThree  \hspace{.1cm} \right]_{\! 4} - \left[ \hspace{.1cm} \SMFour  \hspace{.1cm} \right]_{\! 4}.$$
\end{theorem}

\medskip

\remark Note that the relation of \Cref{thm:4gonRelation} utilizes a nonplanar configuration, the one with the virtual crossing, and one with a $4$-valent vertex.  The pentagon relation of \Cref{thm:5gonRelation} contains a configuration with a quadrilateral.  One could of course, substitute the relation of \Cref{thm:4gonRelation} to obtain a different pentagon relation.  However, the pentagon relation above has the advantage of using only planar trivalent configurations. For example, if a plane graph is a minimal counterexample to the four color theorem, then replacing a pentagon in it with any of the other configurations in the formula will produce a plane graph with $4$-face colorings. 

The proofs of these theorems are accomplished using the technique of  verifying each relation for each cap individually using the BM polynomial, as described in \Cref{thm:capRelation}.  The approach, carried out in Mathematica code attached as auxiliary files to the arXiv version of this paper, is outlined as follows:
\begin{enumerate}
\item First, generate a list of caps:  1024 caps for the pentagon and 60 for the quadrilateral, using the steps described in Section 5.1 of \cite{BM-Reduce}.  This list includes both planar and nonplanar caps.
\item Insert each configuration into the cap and compute the BM polynomial of each capped configuration.  
\item We then set up a system of 1024 equations with 7 unknowns in the case of the pentagon, and a system of 60 equations with 7 unknowns for the quadrilateral, and then solve. These systems of equations have solutions.
\item Since the relation holds for the BM polynomials of those configurations, \Cref{thm:capRelation} implies that it holds for the Penrose polynomials as well, and the proof is complete.
\end{enumerate}

As described in the proof above, writing 1024 equations with 7 unknowns and expecting solutions would appear naive.  In truth, our work in \cite{BM-Reduce} gave us a hunch that there was a solution.  Still, we had to write several programs and run many calculations to arrive at the relations above.  For example, we have alternate programs and methods that prove these relations, some of which we did by hand.  The auxiliary files to the arXiv version of this paper represent the most direct and simplest approach.


\section{Concluding remarks}\label{sec:Relations}
We end with a family of relations on the polynomial that hold for all $n>1$.  Unless otherwise stated, the graphs do not need to be trivalent or planar.  

\begin{theorem}
\label{thm:polynomialRelations}
Let $\Gamma$ be a ribbon graph of an abstract graph $G(V,E)$ in which one of the configurations on the left below appears, and the emanating edges are distinct in $\Gamma$.  Then the configuration can be replaced with the corresponding one on the right and the Penrose polynomials of the two ribbon graphs satisfy the following relations for all $n$:
\begin{align}
\label{eqn:DotRel} \left[ \hspace{.1cm} \IDot \hspace{.1cm} \right]_{\! n} &=  2 \left[ \hspace{.1cm} \Vertical  \hspace{.1cm} \right]_{\! n},\\
\label{eqn:FourVRel} \left[ \hspace{.1cm} \FourVRA \hspace{.1cm} \right]_{\! n} &=  2 (n-1) \left[ \hspace{.1cm} \FourVRB  \hspace{.1cm} \right]_{\! n},\\
\label{eqn:BubbleRel}\left[ \hspace{.1cm} \Bubble \hspace{.1cm} \right]_{\! n} &=  2(n-2) \left[ \hspace{.2cm} \Vertical \hspace{.2cm} \right]_{\! n},\\
\label{eqn:TwistRel} \left[ \hspace{.1cm} \TrivalentVertTwist \hspace{.1cm} \right]_{\! n} &= - \left[ \hspace{.2cm} \TrivalentVert \hspace{.2cm} \right]_{\! n},\\
\label{eqn:TriangleRel} \left[ \hspace{.1cm} \BUTriangle \hspace{.1cm} \right]_{\! n} &=  (n-2) \left[ \hspace{.2cm} \TrivalentVert \hspace{.2cm} \right]_{\! n}.
\end{align}
\end{theorem}

\begin{remark}
The requirement in this theorem (and all theorems in this paper) that the emanating edges are distinct in $\Gamma$ is an important condition. The theorems are sometimes false without it. For example, a single vertex with one loop has Penrose polynomial $n(n-1)$, whereas a loop without a vertex is $n$. However, a plane graph with two vertices and two edges between them has Penrose polynomial $2n(n-1)$ as \Cref{eqn:DotRel} specifies. Note that the condition is not needed for \Cref{eqn:TwistRel} or \Cref{eqn:TriangleRel}.
\end{remark}

Before proving \Cref{thm:polynomialRelations}, we should distinguish these relations from the ones given by Penrose (see \Cref{fig:Penrose-Orig-Relations}).  Penrose's relations were only for $n=3$ and only applied to trivalent graphs.  Moreover, while several of these (e.g. \Cref{eqn:BubbleRel,eqn:TwistRel,eqn:TriangleRel}) agree with his at $n=3$, the relations above are distinguished by the fact that they are computed on the blowup.  The agreement at $n=3$ happens because $3$-edge colorings of the blowup of a trivalent graph are in one-to-one correspondence with $3$-edge colorings of the original graph by \Cref{eqn:TriangleRel}.


\begin{proof}
Each relation is proven using Theorem E\footnote{In fact, we could use \Cref{thm:capRelation} to prove them as well, but the configurations are simple enough that it is not necessary.}  of \cite{BM-Color}, which interprets the Penrose polynomial in terms of counts of proper $n$-face colorings of state graphs.  In particular, any state graph which is not properly colorable, does not contribute to the polynomial, and may be ignored.  For \Cref{eqn:DotRel}, the factor of two results from the fact that there are twice as many states for the ribbon graph on the left that support colorings.  For \Cref{eqn:FourVRel}, note that the all-zero smoothing contains an extra circle, compared to the ribbon graph where the loop is removed.  After coloring everything but that circle, there are $n-1$ colors available to obtain a proper $n$-face coloring.  However, any state with a $1$-smoothing on the loop, cannot be properly colored, and does not contribute to the polynomial.  The factor of two results from applying \Cref{eqn:DotRel} to remove the vertex.  A similar argument applies to \Cref{eqn:BubbleRel}:  in this case, of the four relevant state graphs for the configuration on the left, only the all-zero state graph is colorable.  Moreover, the all-zero state contains one extra circle, incident to both of the other components, hence there are $n-2$ colors available after coloring everything but that circle.  
Again, the factor of two results from applying \Cref{eqn:DotRel} to remove the vertex.

For \Cref{eqn:TwistRel} we observe that one can ``flip over" the vertex to undo the twist, which inserts an extra half-twist in each ribbon incident to the vertex for every state graph in the hypercube of resolution.  Each half-twist gives an overall factor of negative one, which yields the result.

Finally, for the triangle in \Cref{eqn:TriangleRel} we note that only two of the smoothings can potentially be colored:  the all-zero and the all-one. The proof is completed by the following calculation.
\begin{align*}
\left[ \hspace{.1cm} \BUTriangle \hspace{.1cm} \right]_{\! n} &=  \Bigg \llbracket \hspace{.1cm} \VertexZero  \hspace{.1cm} \Bigg \rrbracket_{\! n} - \Bigg \llbracket \hspace{.1cm} \VertexOne  \hspace{.1cm} \Bigg \rrbracket_{\! n}\\
&=  (n-3) \left[ \hspace{.1cm} \TrivalentVert  \hspace{.1cm} \right]_{\! n} - \left[ \hspace{.1cm} \TrivalentVertTwist  \hspace{.1cm} \right]_{\! n}\\
&=  (n-3) \left[ \hspace{.1cm} \TrivalentVert  \hspace{.1cm} \right]_{\! n} + \left[ \hspace{.1cm} \TrivalentVert  \hspace{.1cm} \right]_{\! n}\\
&=  (n-2) \left[ \hspace{.1cm} \TrivalentVert  \hspace{.1cm} \right]_{\! n} 
\end{align*}
\end{proof}

Part of the beauty of \Cref{thm:polynomialRelations} is that the relations apply for all $n$.  In fact, future work will use some of these relations.  Note, however, \Cref{thm:4gonRelation,thm:5gonRelation} only apply at $n=4$.  The relations given in those theorems do not continue to hold, even for $n=5$.  Regardless, it seems natural to ask the following:

\question By choosing appropriate polynomial coefficients for each configuration (instead of constants), or by possibly choosing more or other configurations, are there relations like in \Cref{thm:4gonRelation,thm:5gonRelation} that hold for all $n$ for the quadrilateral or the pentagon?

\end{document}